\documentclass[12pt,draftcls,onecolumn]{IEEEtran} 

\everymath{\displaystyle}
\usepackage{subfigure}
\usepackage{graphicx}
\usepackage{graphics} 
\usepackage{epsfig} 
\usepackage{mathptmx} 
\usepackage{times} 
\usepackage{amsmath} 
\usepackage{amssymb}  
\usepackage{amsthm}
\usepackage{dsfont}
\usepackage{amssymb} 
\usepackage{color}

\newcommand{\beqnum}{\begin{equation}\begin{array}{lcl}}
\newcommand{\eeqnum}{\end{array}\end{equation}}
\newcommand{\beqnom}{\begin{eqnarray}}
\newcommand{\eeqnom}{\end{eqnarray}}
\newcommand{\beqnc}{\begin{center}\begin{eqnarray}}
\newcommand{\eeqnc}{\end{eqnarray}\end{center}}
\newcommand{\beqnlm}{\begin{equation}\vspace{-.5cm}\begin{array}{lll}}
\newcommand{\eeqnlm}{\end{array}\end{equation}}\vspace{-.5cm}
\newcommand{\beq}{\begin{eqnarray*}}
\newcommand{\eeq}{\end{eqnarray*}}
\newcommand{\bef}{\begin{figure}[tbh!]}
\newcommand{\enf}{\end{figure}}
\newcommand{\vep}{\varepsilon}
\newcommand{\fep}{F_\varepsilon}

\newcommand{\R}{\mathbb{R}}
\newcommand{\lf}{\left\lfloor}
\newcommand{\rr}{\right\rceil}
\usepackage{geometry}
\newtheorem{montheo}{\bf Theorem}
\newtheorem{rem}{\bf Remark}
\newtheorem{lemme}{\bf Lemma}
\newtheorem{Proposition}{\bf Proposition}

\newtheorem{defn}{\bf Definition}

\title{\LARGE \bf
Practical stabilization of perturbed integrator chains with unknown bounds\thanks{This research was partially supported by the iCODE Institute, research project of the IDEX Paris-Saclay, and by the Hadamard Mathematics LabEx (LMH) through the grant number ANR-11-LABX-0056-LMH in the ``Programme des Investissements d'Avenir''.}
}

\author{Yacine Chitour, Mohamed Harmouche, Salah Laghrouche
\thanks{Y. Chitour, L2S - Universite Paris XI, CNRS 91192 Gif-sur-Yvette, France. {\tt\small yacine.chitour@lss.supelec.fr}}%
\thanks{M. Harmouche, Actility, Paris, France.
        {\tt\small mohamed.harmouche@actility.com}}%
\thanks{S. Laghrouche, OPERA Laboratory, UTBM, Belfort, France.
{\tt\small salah.laghrouche@utbm.fr}}%
}

\date{}

\begin{document}
\maketitle

\thispagestyle{empty}
\pagestyle{empty}

\begin{abstract}                          
In this paper, we present Lyapunov-based adaptive controllers for the practical (or real) stabilization of a perturbed chain of integrators with bounded uncertainties. We refer to such controllers as Adaptive Higher Order Sliding Mode (AHOSM) controllers since they are designed for nonlinear SISO systems with bounded uncertainties such that the uncertainty bounds are unknown. Our main result states that, given any neighborhood $\cal{N}$ of the origin, we determine a controller insuring, for every uncertainty bounds, that every trajectory of the corresponding closed loop system enters $\cal{N}$ and eventually remains there. The effectiveness of these controllers is illustrated through simulations.
\end{abstract}

\section{Introduction}
Parametric uncertainty in nonlinear dynamic physical systems arises from varying operating conditions and external perturbations that affect the physical characteristics of such systems. The variation limits or the bounds of this uncertainty might be known or unknown. This needs to be considered during control design so that the controller counteracts the effect of variations and guarantees performance under different operating conditions. Sliding mode control (SMC) \cite{Utkin,Slotine1984} is a well-known method for control of nonlinear systems, renowned for its insensitivity to parametric uncertainty and external disturbance. This technique is based on applying discontinuous control on a system which ensures convergence of the output function (sliding variable) in finite time to a manifold of the state-space, called the sliding manifold \cite{Young_Utkin}. In practice, SMC suffers from \emph{chattering}; the phenomenon of finite-frequency, finite-amplitude oscillations in the output which appear because the high-frequency switching excites unmodeled dynamics of the closed loop system \cite{Utkin1999}.  Higher Order Sliding Mode Control (HOSMC) is an effective method for chattering attenuation \cite{Emelyanove}. In this method the discontinuous control is applied on a higher time derivative of the sliding variable, such that not only the sliding variable converges to the origin, but also its higher time derivatives. As the discontinuous control does not act upon the system input directly, chattering is automatically reduced.

Many HOSMC algorithms exist in contemporary literature for control of uncertain nonlinear systems, where the bounds on uncertainty are known. These are robust because they preserve the insensitivity of classical sliding mode, and maintain the performance characteristics of the closed loop system as long as it remains inside the defined uncertainty bounds. Levant for example, has presented a method of designing arbitrary order sliding mode controllers for Single Input Single Output (SISO) systems in \cite{Levant2001}. In his recent works \cite{Levant2003,Levant2005}, homogeneity approach has been used to demonstrate finite time stabilization of the proposed method. Laghrouche et al. \cite{Laghrouche2} have proposed a two part integral sliding mode based control to deal with the finite time stabilization problem and uncertainty rejection problem separately. Dinuzzo et al. have proposed another method in \cite{Dinuzzo09}, where the problem of HOSM has been treated as Robust Fuller's problem. Defoort et al. \cite{Defoort2009} have developed a robust Multi Input Multi Output (MIMO) HOSM controller, using a constructive algorithm with geometric homogeneity based finite time stabilization of a chain of integrators. Harmouche et al. have also presented their homogeneous controller in \cite{HARMOUCHE1} based on the work of Hong \cite{Hong}.

The case where the bounds on uncertainty exist, but are unknown, is an interesting problem in the field of arbitrary HOSMC. In this problem, a good control strategy is expected to have two essential properties: (a) no use of the uncertainty bounds in the stabilizing controller design; (b) avoidance of gain overestimation \cite{Plestan2010}. In recent years, adaptive sliding mode controllers have attracted the interest of many researchers for this case\cite{Ferreira_CST2011, Efimov2011, Bartolini_IMA2013, Moreno2016IJC}. Adaptive gains have been used with success in the past for chattering suppression. For example, Bartolini et al. \cite{Bartolini_unknown} have extended Utkin's concept of equivalent control for second order sliding mode control gain adaptation, to suppress residual oscillations due to digital controllers with time delay. Similarly, an equivalent control based adaptive controller is described in \cite{XU2004}, in which the equivalent control estimation is improved, using double low pass filters. A concise survey of these methods can be found in \cite{Utkin_chatteringanalysis}. Huang et al. \cite{Huang2008} were the first to use dynamic gain adaptation in SMC for the problem of unknown uncertainty bounds. They presented an adaptation law for first order SMC, which depends directly upon the sliding variable; the control gains increase until sliding mode is achieved, and afterwards the gains become constant. This method works without a-priori knowledge of uncertainty bounds, however it does not solve the gain overestimation problem as the gains stabilize at unnecessarily large values. Plestan et al. \cite{Plestan2010,Plestan2012} have overcome this problem by slowly decreasing the gains once sliding mode is achieved. This method establishes real sliding mode (convergence to a neighborhood of the sliding surface). However it does not guarantee that the states would remain inside the neighborhood after convergence; the state actually overshoots in a region around the neighborhood depending on the uncertainties bounds, which is therefore not known. In the field of HOSMC, Shtessel et al. \cite{Shtessel} have presented a Second Order adaptive gain SMC for non-overestimation of the control gains, based on a supertwisting algorithm. An adaptive version of the twisting algorithm is proposed in \cite{Levant2012}, and applied for pneumatic actuator control. The state overshoots in the cases of \cite{Shtessel} and \cite{Levant2012} as well. However, unlike \cite{Plestan2010}, the magnitude is unknown. A Lyapunov-based variable gains super twisting algorithm has also been presented in \cite{Davila2010}. Glumineau et al. \cite{Glumineau_Adaptive} have presented a different approach, based on impulsive sliding mode adaptive control of a double integrator system. The gain of the impulsive control is adapted to minimize the convergence time of the double integrator dynamics. In \cite{Edwards2016} and \cite{Shtessel2013}  continuous AHOSM control algorithms are studied. They are based on reconstruction of equivalent control. All these works insure sliding set convergence to a bounded zone whose size and convergence time depend upon the upper bounds of the perturbations or their derivatives. In particular, if these upper bounds are not a priori known, one cannot prescribe in advance the size of the convergence zone.

It can be noted that all research works avoiding gain overestimation, discussed so far, have yielded real sliding mode. In fact, real sliding mode is the only possibility when the uncertainty bounds are unknown, as the gain dynamics cannot respond immediately to sudden changes in system parameters. 

In this paper, we present Lyapunov-based adaptive controllers for the finite time stabilization of a perturbed chain of integrators with bounded uncertainties. Through a minor extension of the definition (as explained in the next section), we refer to such controllers as Adaptive Higher Order Sliding Mode controllers (AHOSM controllers or AHOSMC).  The proposed adaptive controller guarantees finite time convergence to an adjustable arbitrary neighborhood of origin whose size does not depends upon the upper bounds of the perturbations or their derivatives, i.e., it establishes real HOSM.  The advantage of this adaptive controller design, compared to other controllers mentioned before, is that this controller can be extended to arbitrary order and the adaptation rates are fast in both directions. 
In addition, the state is confined inside the neighborhood after convergence and cannot escape. As a result, there is no state overshoot and no gain overestimation in this controller; and the neighborhood of convergence can be chosen as small as possible independently  of the upper bounds of the perturbations or their derivatives.

The paper is organized as follows: problem formulation and adaptive controllers are presented  in Section 2, simulation results are shown in Section 3. Some concluding remarks are given in Section 4.



\section{Higher Order Sliding Mode Controllers}

If $r$ is a positive integer, the perturbed chain of integrators of length $r$ corresponds to the (uncertain) control system given by 
\beqnum\label{u.l.s.}
\dot z_r =\varphi(t) + \gamma(t)u, 
\eeqnum
where $z=[z_1\  z_2\ ... z_r ]^T\in\mathbb{R}^r$, $u\in\mathbb{R}$ and the functions $\varphi$ and $\gamma$ are any measurable functions defined almost everywhere (a.e. for short) on $\mathbb{R}_+$ and bounded by positive constants $\bar \varphi$, $\gamma_m$ and $\gamma_M$, such that, for a.e. $t\geq 0$, 
\beqnum\label{bound}
\left| \varphi(t) \right| \le \bar \varphi,\quad
0 < \gamma_m\le \gamma(t) \le \gamma_M.
\eeqnum
One can equivalently define a perturbed chain of integrators of length $r$ as the differential inclusion $\dot z_r\in  I_{\bar\varphi}+u I_\gamma$ where $I_{\bar\varphi}=\left[-\bar\varphi , \bar\varphi \right]$ and $ I_\gamma=\left[\gamma_m , \gamma_M \right]$.

The usual objective regarding System \eqref{u.l.s.} consists of stabilizing it with respect to the origin in finite time, i.e., determining feedback laws $u=U(z)$ so that the trajectories of the corresponding closed-loop system converge to the origin in finite time. Note that, in general, the controllers $U(\cdot)$ are discontinuous and then, solutions of  \eqref{u.l.s.} need to be understood here in Filippov's sense \cite{Filippov}, i.e., the right-hand vector set is enlarged at the discontinuity points of the differential inclusion to the convex hull of the set of velocity vectors obtained by approaching $z$ from all directions in $\mathbb{R}^r$, while avoiding zero-measure sets. Several solutions for this problem exist  \cite{levant93,  Levant_Springer, Levant2005, Laghrouche2, Laghrouche_CST}  under the hypothesis that the bounds $\gamma_m,\gamma_M$ and $\bar \varphi$ are known. 

In case the bounds $\gamma_m,\gamma_M$ and $\bar \varphi$ are unknown (one only assumes their existence) then it is obvious to see that  finite time stabilization is not possible by a mere state feedback and therefore, one possible alternate objective consists in achieving practical stabilization. This is the goal of this paper to establish such a result for System \eqref{u.l.s.} and we provide next a precise definition of  practical stabilization.
\begin{defn}\label{def1} We say that $\dot z=f(z,u)$ is practically stabilizable, if for every $\varepsilon>0$, there exists a controller $u=U_\vep(\cdot,t)$ such that every trajectory of the closed-loop system $\dot z=f(z,U_\vep(\cdot,t))$ enters the open ball of radius $\vep$ centered at the origin and eventually remains there. 
\end{defn}
The main result of that paper consists of designing controllers which practically stabilize System \eqref{u.l.s.} {\emph{independently} }of the positive bounds $\bar \varphi$, $\gamma_m$ and $\gamma_M$, i.e., for every $\vep>0$, the controllers $u=U_\vep(\cdot,t)$ which practically stabilize System \eqref{u.l.s.} does not depend on the bounds $\bar \varphi$, $\gamma_m$ and $\gamma_M$.

We next recall the following definition needed in the sequel.  
\begin{defn}\label{homogeneity} ({\bf Homogeneity.} cf. \cite{Levant_Springer}.)
If $r,m$ are positive integers, a function $f:\mathbb{R}^r\rightarrow \mathbb{R}^m$ (or a differential inclusion $F:\mathbb{R}^r\rightrightarrows \mathbb{R}^m$ respectively) is said to be {\it homogeneous of degree $q\in\mathbb{R}$ with respect to the family of dilations $\delta_\vep(z)$}, $\vep>0$, defined by 
$$
\delta_\vep(z)=(z_1,\cdots,z_r)\mapsto (\vep^{p_1}z_1,\cdots,\vep^{p_r}z_r),
$$
where $p_1\cdots,p_r$ are positive real numbers (the weights), if for every positive $\vep$ and $z\in\mathbb{R}^r$, one has $f(\delta_\vep(z))=\vep^qf(z)$  $\big(F(\delta_\vep(z))=\vep^q\delta_\vep(F(z))\ respectively\big)$.
\end{defn}

The following notations will be used throughout the paper.
We define the function $sgn$ as the multivalued function defined on $\mathbb{R}$ by $sgn(x)=\frac x{\vert x\vert}$ for $x\neq 0$ and $sgn(0)=[-1,1]$. Similarly, for every $a\geq 0$ and $x\in \mathbb{R}$, we use $\lf x\rr^a$ to denote $\left| x \right|^a \text{sgn}(x)$. Note that $\lf \cdot\rr^a$ is a  continuous function for $a>0$ and of class $C^1$ with derivative equal to $a\left| \cdot \right|^{a-1}$ for $a\geq 1$. Moreover, for every positive integer $r$, we use $J_r$ to denote the $r$-th Jordan block, i.e., the $r\times r$ matrix whose $(i,j)$-coefficient is equal to $1$ if $i=j-1$ and zero otherwise.

\subsection{Adaptive Higher Order Sliding Mode Controller}\label{adaptatif}
We first define the system under study and provide parameters used later on. 
\begin{defn}\label{def0} {\it Let $r$ be a positive integer. The $r$-th order chain of integrator
$(CI)_r$ is the single-input control system given by 
\beqnum \label{pure_int}
(CI)_r\ \ \ \dot z=J_rz+u e_r,
\eeqnum
with $z=(z_1,\cdots,z_r)^T\in\mathbb{R}^r$ and $u\in \mathbb{R}$.
For $\kappa<0$ and $p>0$ with $p+(r+1)\kappa\in [0,1)$, set $p_i := p + (i-1)\kappa, \ 1\leq i\leq r+2$. For $\varepsilon>0$, let $\delta_\vep:\R^r\rightarrow \R^r$ be the family of dilations associated with $\left( p_1, \cdots , p_r \right)$.}
\end{defn}
In the spirit of \cite{Laghrouche_CST, Harmouche_CDC12}, we put forwards geometric conditions on certain stabilizing feedbacks $u_0(\cdot)$ for $(CI)_r$ and corresponding Lyapunov functions $V_1$. These conditions will be instrumental for the latter developments. 

Our construction of the feedback for practical stabilization relies on the following result. 
%
%
\begin{montheo}\label{theo1}
Let $r$ be a positive integer. There exists a feedback law 
$u_0:\R^r\rightarrow\R$ 
homogeneous  with  respect to $(\delta_\vep)_{\vep>0}$ such that the closed-loop system $\dot z=J_rz+u_0(z)e_r$ is finite time globally asymptotically stable with respect to the origin and the following conditions hold true:
\begin{description}
	\item[$(i)$] the function $z\mapsto J_rz+u_0(z)e_r$ is homogeneous of degree $\kappa$ with respect to $(\delta_\vep)_{\vep>0}$ and there exists a continuous positive definite function $V_1:\R^r\rightarrow \R_+$, $C^1$ except at the origin, homogeneous with respect to $(\delta_\vep)_{\vep>0}$ such that there exists $c>0$ and $\alpha\in (0,1)$ for which the time derivative of $V_1$ along non trivial trajectories of $\dot z=J_rz+u_0(z)e_r$ verifies 
	\beqnum \label{est-V1}
	\dot V_1 \le -cV_1^\alpha;
	\eeqnum
		\item[$(ii)$] the function $z\mapsto u_0(z)\partial_r V_1(z)$ is non positive over $\mathbb{R}^r$ and, for every non zero $z\in \mathbb{R}^r$ verifying $u_0(z)=0$, one has $\partial_r V_1(z)=0$. As a consequence function $z\mapsto sgn(u_0(z))\partial_r V_1(z)$
		is well-defined over $\mathbb{R}^r$ and non positive.
\end{description}
\end{montheo}
%
%
\begin{rem}
Item $(i)$ of the above theorem is classical, see for instance \cite{Hong,Huang2005,Zavala_2014,arxiv2013}. Item $(ii)$ considers a geometric condition on controllers verifying Item $(i)$, which was introduced in \cite{Harmouche_CDC12} and used in \cite{Laghrouche_CST, CHL15}. This geometric condition is indeed satisfied, for instance by Hong's controller, see \cite{CHL15} for other examples.
\end{rem}

Regarding our problem, we consider, for every $\vep>0$ the following controller:
\beqnum \label{eq_v}
u_\vep(z,t) = g(\vert u_0(z)\vert) u_0(z) +k\, \hbox{sgn}\big( u_0(z) )\hat \varphi_\vep (t,V_1(z)\big),
\eeqnum
where  $u_0$ and $V_1$ are provided by Theorem~\ref{theo1}, $g:\mathbb{R}_+\rightarrow \mathbb{R}_+^*$ is an increasing $C^1$ function with $\lim_{x\rightarrow +\infty}g(x)=+\infty$ and 
the adaptive function $\hat \varphi_\vep$ is defined as  
\beqnum
		\hat \varphi_\vep(t,x) &=& \left\{
		\begin{array}{ccc}
\min\big(t, F_\varepsilon(x)\big), & \text{ if } 0\leq x< \varepsilon ,& \\
t,  &\text{ if }x\geq \varepsilon,& 
\end{array}
\right.
\eeqnum
with $\fep(x)=\frac{\vep}{\vep-x}$ for $x\in [0,\vep)$. Here the positive function $g$ and the positive constant $k$ are gain parameters whereas $\vep>0$ will be used to define the arbitrarily small neighborhood where trajectories of \eqref{u.l.s.} with feedback control law \eqref{eq_v} will eventually end up.   

The following theorem provides the main result for the adaptive controller $u_\vep$.
\begin{montheo}\label{theorem_adaptatif}
Let $r$ be a positive integer and System \eqref{u.l.s.} be the perturbed $r$-chain of integrators with unknown bounds $\gamma_m,\gamma_M$ and $\bar{\varphi}$. Let $\vep>0$ and $u_0,V_1:\mathbb{R}^r\rightarrow \mathbb{R}_+$ be the feedback law and the continuous positive definite function defined respectively in Theorem~\ref{theo1}. Then, for every trajectory $z(\cdot)$ of the closed-loop system \eqref{u.l.s.} under the feedback control law \eqref{eq_v}, one has 
\beqnum\label{mainR}
\limsup_{t\rightarrow \infty}V_1(z(t)) \leq \max\big(0,\overline{V}_1\big),
\eeqnum
where $\overline{V}_1:=\vep(1-\frac1{\bar\Phi})$, with $\bar \Phi := \frac1{k\gamma_m}\left( \bar \varphi-h_m\right)$, where $h_m=\min\big(0,\min_{x\geq 0}(\gamma_mg(x)-1)x\big)$.
\end{montheo}
%
%
%
%
\subsection{Proof of Theorem \ref{theorem_adaptatif}}
We refer to $(S)$ as the closed-loop system defined by  \eqref{u.l.s.} and \eqref{eq_v}. 
The first issue we address is the existence of trajectories of $(S)$ starting at any initial condition $z_0\in\mathbb{R}^r$. Such an existence follows from the fact that the application $\mathbb{R}_+\times \mathbb{R}^r\rightarrow \mathbb{R}$, $(t,z)\mapsto \hat \varphi_\vep (t,V_1(z))$ is continuous. 

We next show that every trajectory of $(S)$ is defined for all positive times. For that purpose, consider a non trivial trajectory $z(\cdot)$ and let $I_{z(\cdot)}$ be its (non trivial) domain of definition.  We obtain the following inequality for the time derivative of  $V_1(z(\cdot))$ on  $I_{z(\cdot)}$ by using Items $(i)$ and $(ii)$ of Theorem~\ref{theo1}. For a.e. $t\in I_{z(\cdot)}$, one gets
%
\beqnum \label{dot_V1}
	\dot V_1 &=&  \frac{\partial V_1}{\partial z_1 }z_2 + ... + \frac{\partial V_1}{\partial z_r }  \left( \gamma  \left[ g(\vert u_0\vert)  u_0 +k\, \text{sgn} (u_0) \hat \varphi_\vep  \right] + \varphi \right),\\
	&\le& -c V_1 ^\alpha \!-\! \!\left|\frac{\partial V_1}{\partial z_r } \right| \!\! \Big( \! (\gamma_m g(\vert u_0\vert)-1) |u_0|  \!+\! k\gamma_m \hat \varphi_\vep \!-\! \bar \varphi \Big),\\
&\le& -c V_1 ^\alpha - k\gamma_m\left|\frac{\partial V_1}{\partial z_r } \right| \left( \hat \varphi_\vep - \bar \Phi\right).
\eeqnum
We thus have the differential inequality a.e. for $t\in I_{z(\cdot)}$
\beqnum\label{dot_V1-1}
\dot V_1 \leq -c V_1 ^\alpha +C_1V_1^{p_{r+2}},
\eeqnum
where $C_1$ is a positive constant independent of the trajectory $z(\cdot)$. Since $p_{r+2}\in [0,1)$, it is therefore immediate to deduce that there is no blow-up in finite time and thus $I_{z(\cdot)}=\mathbb{R}_+$.  

We now prove that Eq.~\eqref{mainR} holds true for any trajectory  $z(\cdot)$ of $(S)$. Assume first that $1\geq \bar\Phi$. Then, for $t>\bar \Phi$, one has that $ \hat \varphi_\vep(t)\geq \bar\Phi$ since $F_\vep$ takes values larger than $1$. It implies that $\dot V_1\leq -c V_1 ^\alpha$ and thus one gets convergence to zero in finite time. 
Assume next that $1<\bar\Phi$. Set $\overline{V}_1:=\vep(1-\frac1{\bar\Phi})$ and notice that $\fep(\overline{V}_1)=\bar\Phi$. In that case, for $t>\bar \Phi$, Eq.~\eqref{dot_V1} can be written
\beqnum\label{eq:V_1-1}
\dot V_1\leq -c V_1 ^\alpha -k\gamma_m\left|\frac{\partial V_1}{\partial z_r } \right| \min(0, \fep(V_1(z))-\fep(\overline{V}_1)).
\eeqnum
Taking into account the fact that $F_\vep$ is an increasing function on $[0,\vep)$, one deduces that there exists at most one time $\bar{t}>1$ such that $V_1(z(\bar{t}))=\overline{V}_1$ (since $\dot V_1(\bar{t})\leq -c\overline{V_1}^\alpha<0$) and if it exists, then 
  $V_1(z(t))<\overline{V_1}$ for $t>\bar{t}$. If such a time $\bar{t}$ does not exists, then one has necessarily $V_1(z(t))<\overline{V_1}$ for $t>\bar \Phi$ since the other alternative would yield convergence in finite time and thus a contradiction. 
  
\begin{rem} In case the controller $u_0$ is bounded, one can remove the assumption that $\lim_{x\rightarrow +\infty}g(x)=+\infty$.
\end{rem}

\subsection{Asymptotic bounds for the controller $u$ and the convergence time to the neighborhood ${\cal{V}}_\vep$}
One deduces from Theorem~\ref{theorem_adaptatif} the following two
 results. The first one is immediate and 
 provides an asymptotic upper bound for the controller $u$.
\begin{lemme}
For $\vep>0$, the controller $u_\vep$ defined in \eqref{eq_v} verifies the following asymptotic upper bound, which is uniform with respect to trajectories of the closed-loop system $(S)$: if $1\geq \bar\Phi$, then
$\lim_{t\rightarrow \infty}{|u_\vep|}=1$ and if $1<\bar\Phi$, then
\beqnum
\limsup_{t\rightarrow \infty}{|u_\vep|} \le U_0(\overline{V}_1)g(U_0(\overline{V}_1))+k\bar \Phi,
\eeqnum
where 
$U_0(\xi):=\max_{\{z\, \vert\,  V_1(z) \le \xi \}} |u_0(z)|$ for $\xi\geq 0$.
\end{lemme}
For $\vep>0$, define  the open neighborhood ${\cal{V}}_\vep$ of the origin as the set of points $z\in\mathbb{R}^r$ such that $V_1(z)<\vep$. Our second result  provides an asymptotic upper bound for the time needed by any trajectory
of the closed-loop system $(S)$ to eventually enter ${\cal{V}}_\vep$ and remain  inside. 
%

\begin{Proposition}\label{prop:temps}
For $z_0\in\mathbb{R}^r$, the convergence time $T_{z_0}$ needed to reach the open neighborhood ${\cal{V}}_\vep$ of the origin and stay inside verifies the following bound: 
\beqnum\label{eq:Tz0}
T_{z_0}\leq \bar \Phi+\frac{\Big(V_1^{1-p_{r+2}}(z_0)+(1-p_{r+2})C_1\bar\Phi\Big)^{\frac{1-\alpha}{1-p_{r+2}}}}{c(1-\alpha)},
\eeqnum
where the constants $c,\alpha$ are provided by Theorem~\ref{theo1} and the constants $\beta,C_1$ are defined in Eq.~\eqref{dot_V1-1}.
\end{Proposition}
\begin{proof}
For $t\leq \bar \Phi$, one deduces from Eq.~\eqref{dot_V1-1} the upper bound 
$$
V_1(z( \Phi))\leq\Big( V_1^{1-p_{r+2}}(z_0)+(1-p_{r+2})C_1\bar\Phi\Big)^{\frac1{1-p_{r+2}}}.
$$ 
For $t\geq \bar \Phi$, $V_1$ verifies either 
Eq.~\eqref{est-V1} or Eq.~\eqref{eq:V_1-1}, which reduces to Eq.~\eqref{est-V1} if $z(t)\notin{\cal{V}}_\vep$. It is then clear that the right-hand side of Eq.~\eqref{eq:Tz0} is an upper bound for $T_{z_0}$. 
\end{proof}

\section{Simulation Results}\label{simulation}
The performances of the proposed control law are studied next through simulation.
In this section we will perform simulations for uncertain systems of order one and three.

\subsection{Simulation for first order system}
Consider the first order system
\beq
	\dot z_1 = \varphi + \gamma u,
\eeq
where $\gamma$ and $\varphi$ are \textit{discontinuous} bounded uncertainties defined as
\beqnum
	\varphi = 5 \text{sgn}(cos(t)) - 10 \sin(2t),
	\gamma = 3 + 2\text{sgn}(\sin(3t)).
\eeqnum
One can see that
\beq
		0 < 1 \le \gamma \le 5, \ |\varphi| \le 15.
\eeq
The candidate $u_0$ and its related Lyapunov function are given next as:
\beq
	u_0 = -\text{sgn}(z_1),\quad V_1 = |z_1|.
\eeq
The control objective consists to force $z_1$ to the neighborhood of zero defined by $V_1 :=|z_1| \le \epsilon :=0.1$. Based on the given $u_0$ with simple computation, the controller $u$ can be defined as
\beqnum
	u = -\left( 1+\varphi_\varepsilon(t)  \right) \text{sgn}(z_1).
\eeqnum
The performance of the proposed controller with respect to the uncertainties is presented in Figure \ref{ordre1}. 
On can see in Figure \ref{ordre1_s}, the convergence of the state $z_1$ the predefined neighborhood of zero. The control objective is satisfied without overestimation of the controller as seen in Figure \ref{ordre1_u}.
\begin{figure}[htbp!]
\centering
\subfigure[Control law $u$]{
    \includegraphics[width= 8 cm]{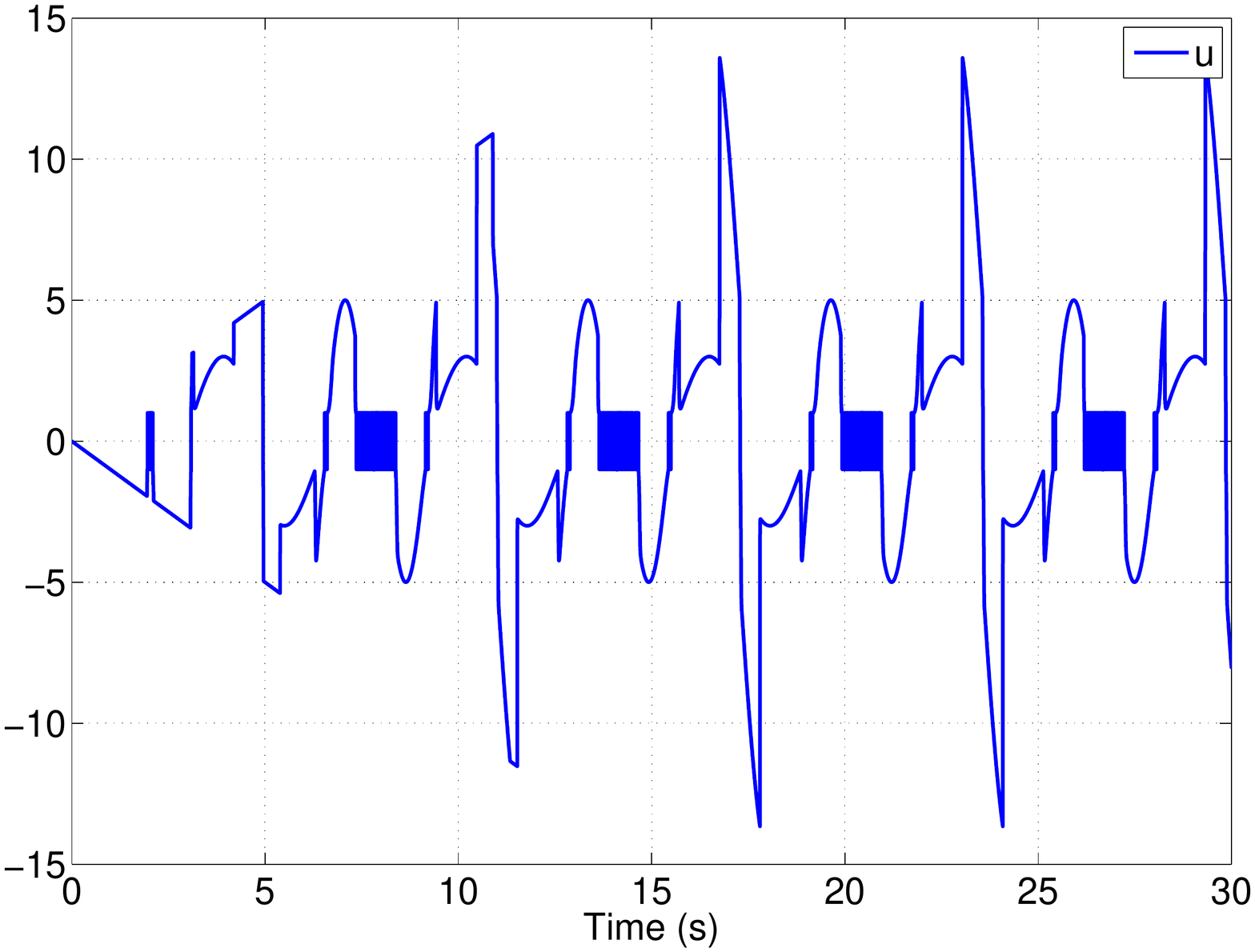}
    \label{ordre1_u}
}\hspace{-1cm}
\subfigure[State $z_1$]{
    \includegraphics[width= 8 cm]{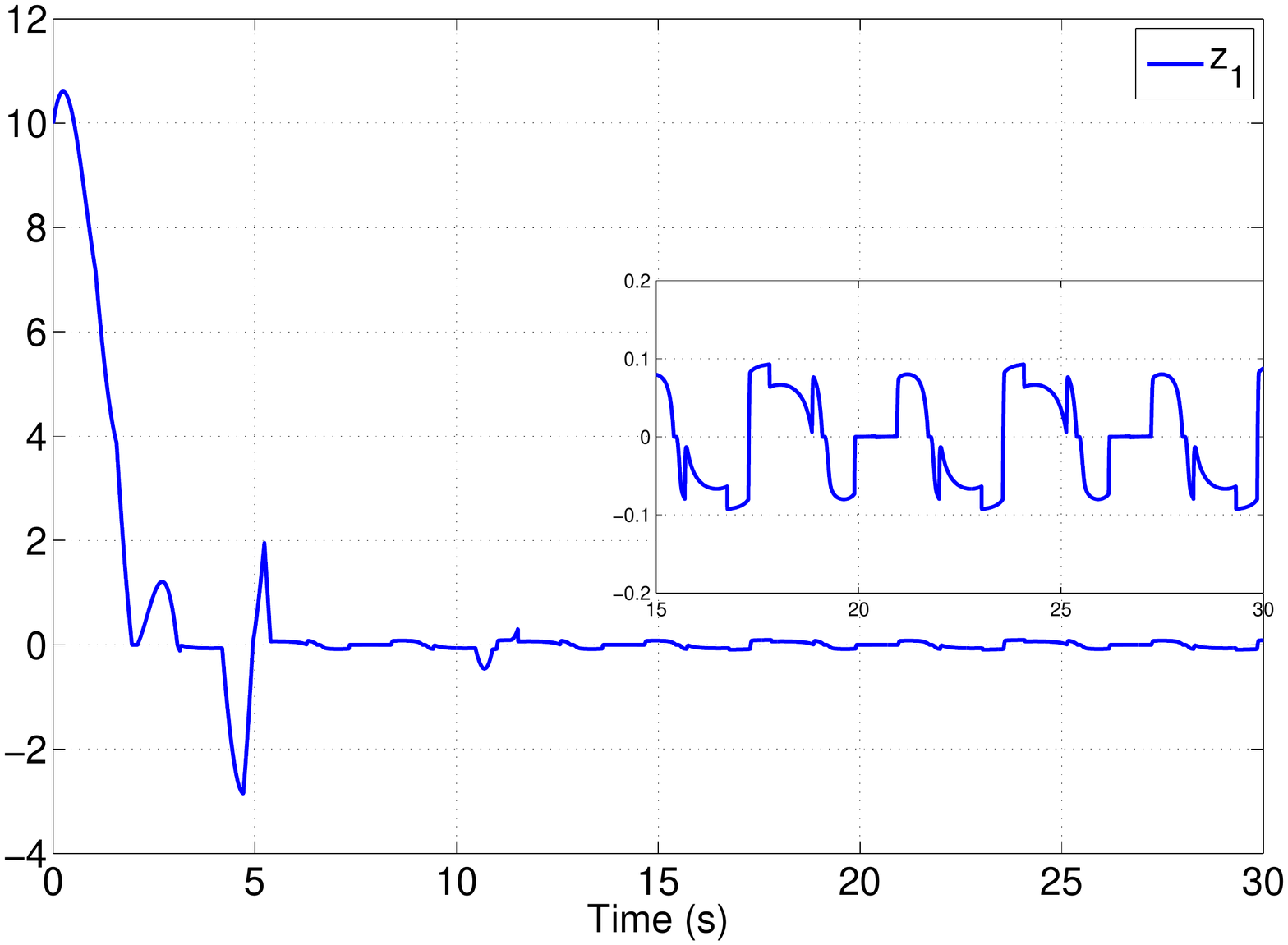}
    \label{ordre1_s}
}\hspace{-1cm}
\caption{Simulation for order 1}
\label{ordre1}
\end{figure}

\subsection{Simulation for third order system}

For arbitrary order, we can refer to Hong's controller \cite{Hong2005} as candidate for the controller $u_0$ defined as follows. Let $\kappa<0$ and $l_1,\cdots,l_r$ positive real numbers. For $z=(z_1,\cdots,z_r)$, we define for $ i=0,...,r-1$:
\beqnum\label{Hongfunction}
p_i = 1+(i-1)\kappa,\\ v_0=0,\,
v_{i+1} = -l_{i+1} \lfloor\lfloor z_{i+1} \rceil^{\beta_i } - \lfloor v_i \rceil^{\beta_i } \rceil^{(\alpha_{i+1}/(\beta_i)},\\
u_0 = v_r,
\eeqnum
where $\alpha_i=\frac{p_{i+1}}{p_i}$.\\
The Lyapunov Function candidate $V_1$ is defined in the following form:
\beqnum\label{Lyap-Hong}
{V_1} = \sum\limits_{j = 1}^r \int\limits_{{v_{j - 1}}}^{{z_j}} {{{\left\lfloor s \right\rceil }^{{\beta _{j - 1}}}} - {{\left\lfloor {{v_{j - 1}}} \right\rceil }^{{\beta _{j - 1}}}}ds}.
\eeqnum

In this section, we consider the following third order system
\beq
	\dot z_1 &=& z_2,\\
	\dot z_2 &=&  z_3,\\
	\dot z_3 &=& \varphi + \gamma u,
\eeq
with the same uncertainty as given in previous simulations. In this  simulation, the control parameters of $u_0$ have been tuned to the following values
\beqnum
		l_1=1,\ l_2=2,\ l_3 = 5,\ \kappa = -1/4.
\eeqnum
The function $g(u_0)$ has been taken as 
\beqnum
		g(u_0) = 1 + \log(1+|u_0|).
\eeqnum
The control objective consists to force the states $z_1,\ z_2, \ z_3$ to a neighborhood of zero defined by $\left\{z=(z_1,z_2,z_3) : V_1(z_1,z_2,z_3) \le 0.01\right\}$. In Figure \ref{ordre3_s}, one can see the practical convergence of $z_1$, $z_2$ and $z_3$. The control objective is achieved as seen in Figure \ref{ordre3_LLV1}, where $LLV1 := \log (1 + V_1)$. The controller and the adaptive gain are presented in Figure \ref{ordre3_u} and Figure \ref{ordre3_k} respectively.
\begin{figure}[htbp!]
\centering
\subfigure[Control law $u$]{
    \includegraphics[width= 8 cm]{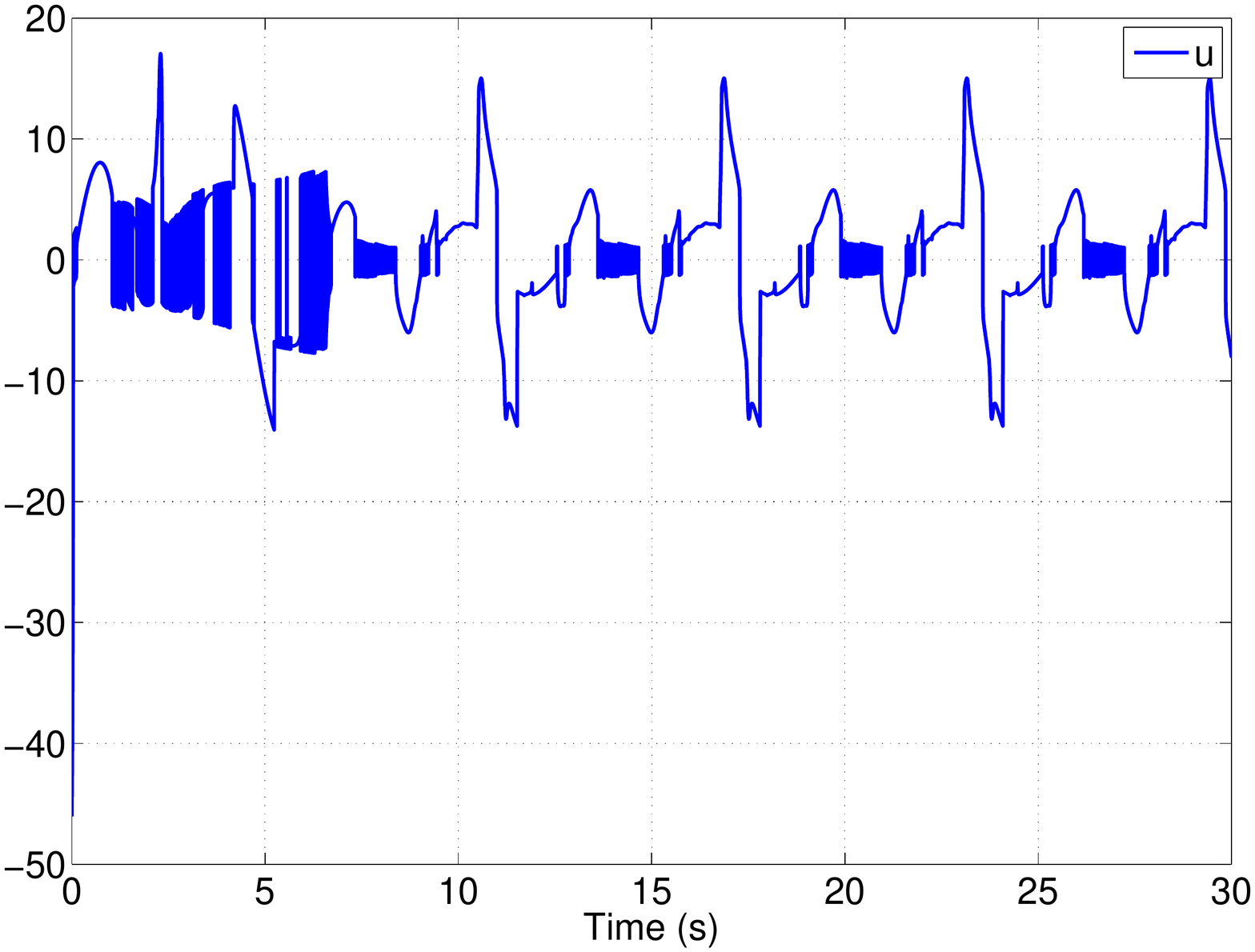}
    \label{ordre3_u}
}\hspace{-1cm}
\subfigure[State $z_1$]{
    \includegraphics[width= 8 cm]{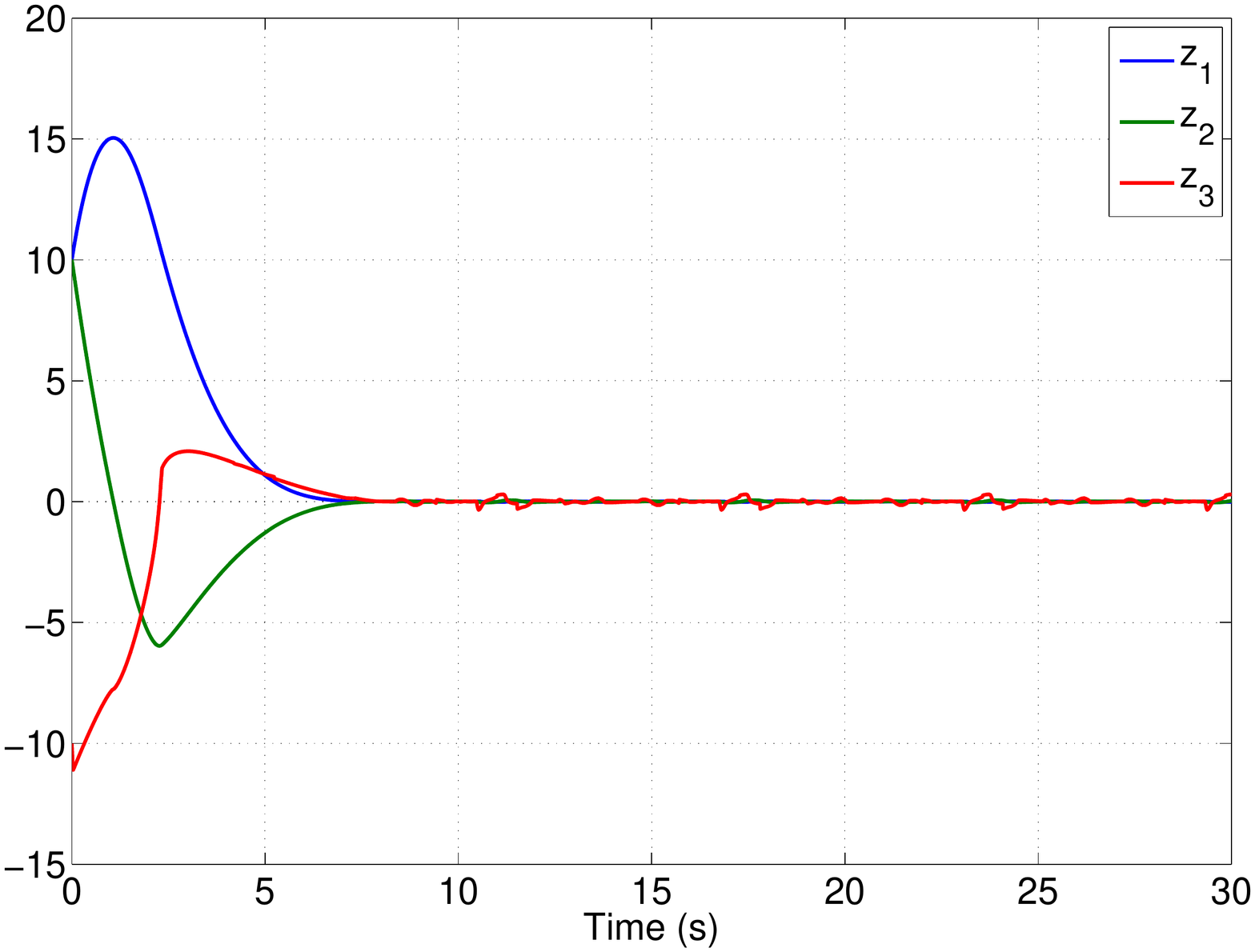}
    \label{ordre3_s}
}\hspace{-1cm}
\subfigure[ Adaptive gain $\hat\varphi_\varepsilon$]{
    \includegraphics[width= 8 cm]{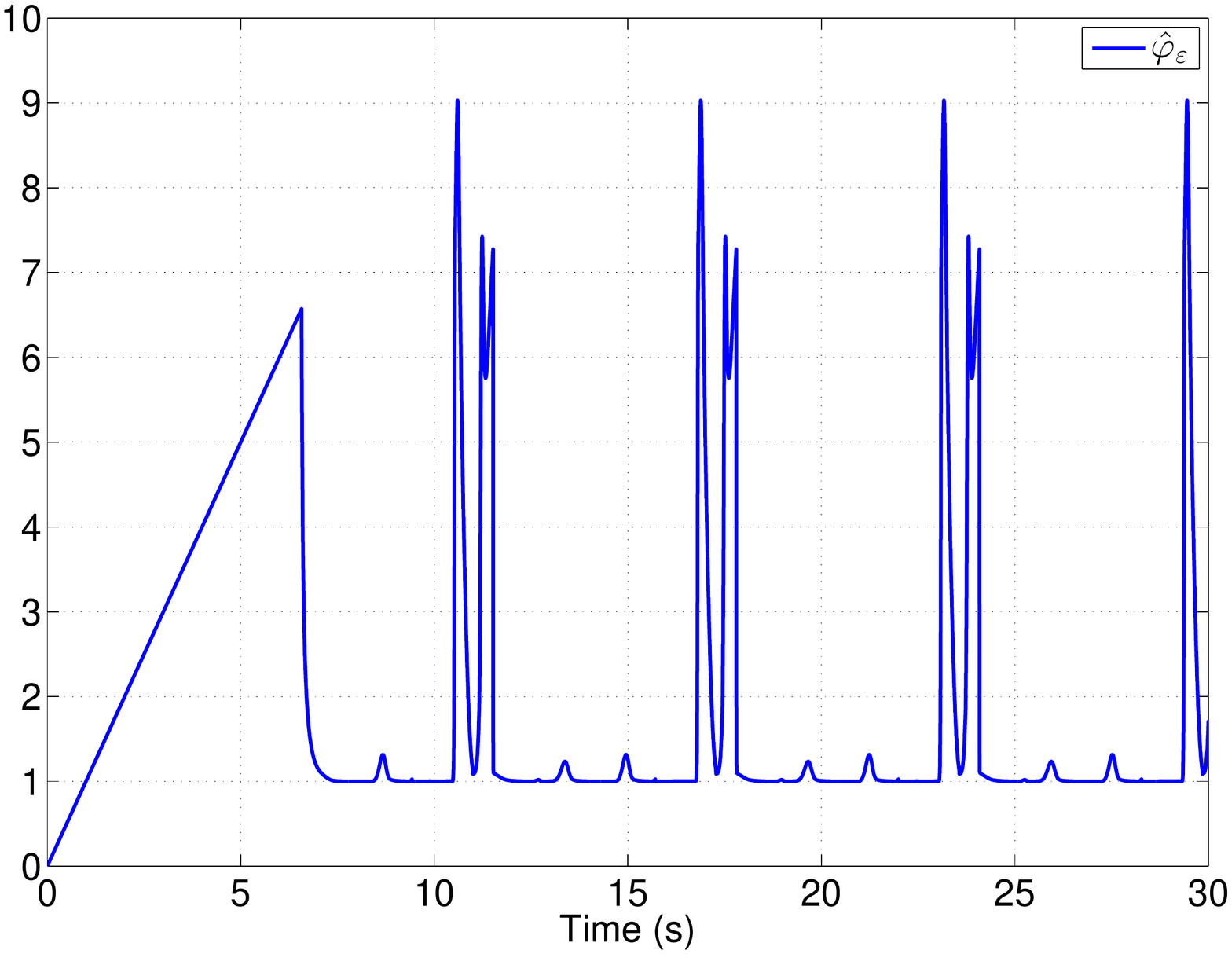}
    \label{ordre3_k}
}\hspace{-1cm}
\subfigure[ Lyapunov Function]{
    \includegraphics[width= 8 cm]{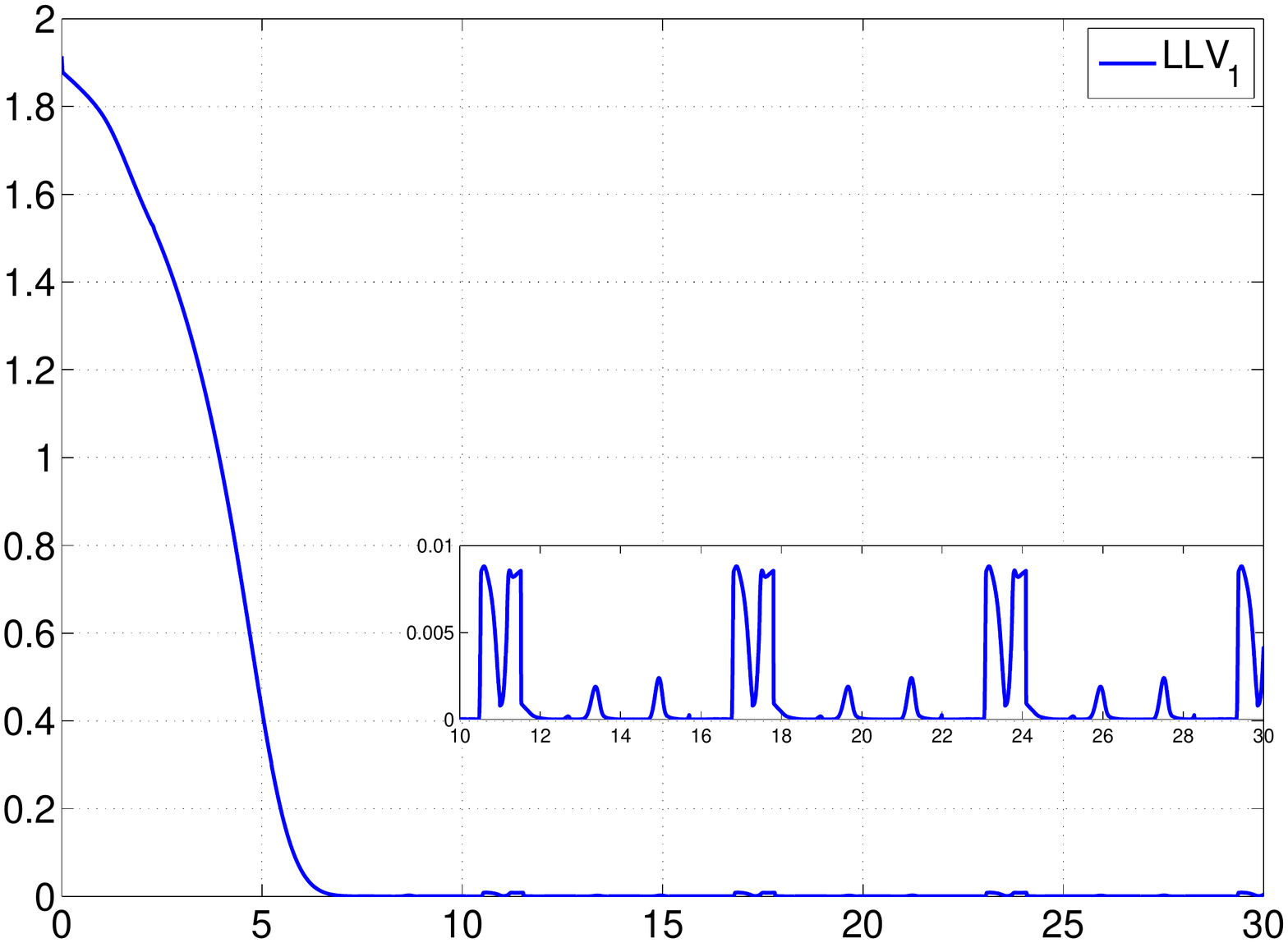}
    \label{ordre3_LLV1}
}
\caption{Simulation for order 3}
\label{ordre3}
\end{figure}

\section{Conclusions}
\noindent This paper has proposed a new Lyapunov-based adaptive scheme for higher-order sliding mode controller with bounded unknown uncertainties. The proposed adaptive controller guarantees finite time convergence to an adjustable arbitrary neighborhood of origin.  The advantage of this adaptive controller, compared to others, is that this controller can be extended to any arbitrary order. In addition, the state is confined inside the neighborhood after convergence and cannot escape. As a result, there is no state overshoot and no gain overestimation in this controller; and the neighborhood of convergence can be chosen as small as possible.

\bibliography{References_Stabilisation_2016_06_28}

\end{document}